\documentclass[10pt]{article}      

\usepackage{amsmath,amssymb, amsfonts,amsthm,amscd} 
\usepackage{bbold, dsfont, mathrsfs, mathabx}
\usepackage{nccmath}   
\usepackage[hidelinks]{hyperref}
\usepackage{enumerate}  

\theoremstyle{plain}
\newtheorem{theorem}{Theorem}[section]
\newtheorem{lemma}[theorem]{Lemma}

\newtheorem{corollary}[theorem]{Corollary}

\theoremstyle{definition}

\theoremstyle{remark}
\AtEndEnvironment{remark}{\qed}

\DeclareMathOperator*{\sumsum}{\sum\!\sum}

\def\bigmid{\,\big{|}\,}

\def\cond{\mathrm{cond}}

\def\div{\mathop{\mathrm{div}}\nolimits}
\newcommand{\grad}{\nabla}

\def\A{\mathcal{A}}

\def\EE{\mathbb E}
\def\I{\mathcal{I}}

\def\M{\mathcal{M}}
\def\NN{\mathbb{N}}
\def\P{\mathcal{P}}
\def\PP{\mathbb{P}}
\def\Q{\mathcal{Q}}
\def\R{\mathcal{R}}
\def\RR{\mathbb{R}}

\def\X{\mathcal{X}}

\date \today
\title{Markov-bridge representation of ergodic large-deviation principles}
\author{D.R.\ Michiel Renger\thanks{TU M\"unchen, Boltzmannstra{\ss}e 3, 85747 Garching, Germany. Email: \href{mailto:d.r.m.renger@tum.de}{d.r.m.renger@tum.de}}}

\begin{document}
\maketitle
\begin{abstract} We revisit classic ergodic large-deviation principles: for the occupation measure (Donsker-Varadhan), and for the empirical flux. We show that these problems can be embedded into a more general, discrete-time framework. A conditioning and mixing argument then yields alternative expressions for these well-known rate functionals, formulated in terms of Markov bridges.
\end{abstract}

\section{Introduction}

\subsection*{Ergodic large-deviation principles}

Consider a stochastic process $(A(t))_{t\geq0}$ on $\RR^d$ and its ergodic average $\bar A_T:=T^{-1}\int_0^T\!A(t)\,dt$. Under sufficient ergodicity assumptions this average converges as $T\to\infty$ to a deterministic ergodic limit. We revisit the classic problem of deriving the corresponding large-deviation principle~\cite{DemboZeitouni09}, formally the exponential rate of decay: 
\begin{equation}
  \PP(\bar A_T \approx a) \sim e^{-T I(a)}, \qquad\text{ as } T\to\infty.
\label{eq:formal LDP}
\end{equation}
Here $I(a)$ is called the ``rate functional'' and is minimised by the ergodic limit. 

Our work is motivated by the following classic results:
\begin{enumerate}[(A)]
\item\label{it:DVG} $A(t):=\mathds1_{X(t)}$ for an irreducible continuous-time Markov chain on a finite space $\X$ with generator matrix $Q$. In this case $\bar A_T=T^{-1}\int_0^T\mathds1_{X(t)}\,dt\in\P(\X)$ is simply the occupation measure, signifying the proportion of time that the chain $(X(t))_{t\geq0}$ spends in each state $x\in\X$. The classic result due to Donsker, Varadhan and G{\"a}rtner shows that the large-deviation principle~\eqref{eq:formal LDP} holds with rate functional~\cite{DonskerVaradhanI,DonskerVaradhanII,DonskerVaradhanIII,DonskerVaradhanIV,Gartner1977}:
\begin{equation}
  I_\mathrm{DVG}(\rho):=\sup_{u\in\RR^\X, \min_x(u_x)>0} -\sum_{x\in\X}\!\frac{(Q u)_x}{u_x} \rho_x,
\label{eq:DVG}
\end{equation}
and $I_\mathrm{DVG}$ is minimised by the invariant measure $\pi\in\P(\X)$ of the chain $(X(t))_{t\geq0}$.

\item In the same model as in \eqref{it:DVG}, the cumulative empirical flux
\begin{equation}
  W_{xy}(t):=\sum_{s\in\lbrack0,t\rbrack:X(s^-)\neq X(s)}\mathds1_{(X(s^-),X(s))}(x,y)
\label{eq:cumflux}
\end{equation}
counts the number of jumps $x$ to $y$ in time interval $\lbrack0,t\rbrack$. If $A(t):=(\mathds1_{X(t)},\dot W(t))$ \footnote{Strictly speaking the time derivative $\dot W(dt)$ exists as a singular measure in time, but since $A(t)$ appears as integrands only, we allow this minor abuse of notation in the introduction.}, then the large-deviation principle~\eqref{eq:formal LDP} holds with rate functional~\cite{BFG2015a,BFG2015b,BCFG2018}:
\begin{align}
  I_\mathrm{BFG}(\rho,j) &:=\begin{cases} \sumsum_{x,y\in\X} s(j_{xy}\mid \rho_x Q_{xy}),  & \div j =0, j\ll\rho\otimes Q,\\
                               \infty,                &\text{else,}
                 \end{cases},
  \label{eq:BFG}\\
  s(a\mid b) &:=
  \begin{cases}
    a \log\frac{a}{b} - a + b, &a,b>0,\\
    b,                         &a=0, b\geq0,\\
    \infty,                    &\text{otherwise},
  \end{cases}
  \label{eq:s}
\end{align}
and $\I_\mathrm{BFG}$ is minimised by $\rho_x=\pi_x$, $j_{xy}=(\rho\otimes Q)_{xy}:=\pi_xQ_{xy}$.

\item From the contraction principle one immediately recovers two separate large-deviation principles~\eqref{eq:formal LDP} for the occupation measure (as above) and the average empirical flux $T^{-1}\int_0^T\dot W(t)\,dt=T^{-1}W(T)$ respectively:
\begin{align*}
  I_\mathrm{DVG}(\rho)=\inf_{j\in\RR^{\X\times\X}} I_\mathrm{BFG}(\rho,j),
  &&
  I_\mathrm{flux}(j):=\inf_{\rho\in\P(\X)} I_\mathrm{BFG}(\rho,j).
\end{align*}
\end{enumerate}
Comparing these results shows that -- as common in large-deviation theory -- rate functionals are often only implicitly defined, but they may become explicit after including more variables in the description. 

In the current paper we derive a similar large-deviation principle as above with an explicit rate functional, obtained by including \emph{different} variables. Inspired by~\cite{BCFG2018}, the main argument stems from rewriting $T=n T_0$ for a \emph{fixed} $T_0>0$ and $n\in\NN$, so that the ergodic average becomes \footnote{The restriction to discrete values of $T=nT_0$ for $n\in\NN$ does not influence the large deviations, as a straightforward Markov inequality shows that $\bar A_{nT_0}$ and $\bar A_{(n+\alpha)T_0}$ for $\alpha\in(0,1)$ are exponentially equivalent as soon as $\int_0^{\alpha T_0}\!A(t)\,dt$ has finite mean.},
\begin{align}
  \bar A^n:=\bar A_{nT_0}=\frac1n \sum_{m=1}^n \frac1{T_0} \int_{(m-1)T_0}^{mT_0}\!A(t)\,dt, 
\label{eq:ergodic average rewritten1}
\end{align}
and then considering the limit $n\to\infty$.

In all examples above, the variables $T_0^{-1} \int_{(m-1)T_0}^{mT_0}\!A(t)\,dt$ in \eqref{eq:ergodic average rewritten1} become independent after conditioning on the values $(X((m-1)T_0),X(mT_0))$ of another but related Markov process $(X(t))_{t\geq0}$ on a finite set $\X$. Thus we might as well consider the coupled, \emph{discrete-time} process 
\begin{align}
  X_m:=X(mT_0), &&
  A_m:=\frac{1}{T_0} \int_{(m-1)T_0}^{mT_0}\!A(t)\,dt.
\label{eq:discrete-time process}
\end{align}

\subsection*{Result in discrete time}

From here we focus on a \emph{general} homogeneous and irreducible discrete-time Markov chain $(X_m)_{m\in\NN_0}$ on a finite set $\X$ and another discrete-time process $(A_m)_{m\in\NN}$ on $\RR^d$. Define:
\begin{align}
  q^{xy}(da)     &:=\PP(A_1\in da\mid X_0=x, X_1=y), 
  \label{eq:q phi phi*}\\
  \phi^{xy}(\lambda) &:=\log\int\! e^{\lambda\cdot a}\,q^{xy}(da), 
  &
  \phi^{xy*}(a)     &:= \sup_{\lambda\in\RR^d} \big\lbrack \lambda\cdot a - \phi^{xy*}(\lambda) \big\rbrack, \notag\\
  \phi_{\lvert\cdot\rvert}^{xy}(s)&:=\log\int\!e^{s\lvert a\rvert_1}\,q^{xy}(da),
  &
  \phi_{\lvert\cdot\rvert}^{xy*}(r)&:= \sup_{s\in\RR} \big\lbrack rs - \phi_{\lvert\cdot\rvert}^{xy}(s) \big\rbrack.
  \notag
\end{align}
The conditional independence mentioned above can be exploited by studying the two variables:
\begin{align*}
  K^n:=\frac1n \sum_{m=1}^n \mathds1_{(X_{m-1},X_m)}\, A_m ,&&
  \Theta^n:=\frac1n\sum_{m=1}^n \mathds1_{(X_{m-1},X_m)}
\end{align*}
in $(\RR^{d})^{\X\times\X}$ and in the probability measures $\P(\X\times\X)\subset\RR^{\X\times\X}$ respectively. Of course $\sum_{x,y\in\X} K^{n,xy} = \bar A^n:=n^{-1}\sum_{m=1}^n A_m$, and so $(K^n,\Theta^n)$ indeed contains more information than $\bar{A}^n$. 

\begin{theorem} 
Assume that:
\begin{align}
  &\PP\big(A_1\in da_1, \hdots, A_n\in da_n \mid X_0=x_0, \hdots, X_n=x_n\big)
  =\prod_{m=1}^n q_{x_{m-1},x_m}(da_m),
\label{asseq:indep}\\
  &\liminf_{r\to\infty} \frac{\phi_{\lvert\cdot\rvert}^{xy*}(r)}{r}=\infty \text{ for all } x,y\in\X.
\label{asseq:superlinear}
\end{align}
Then the sequence $(K^n,\Theta^n)$ satisfies the large-deviation principle
with good rate functional:
\begin{equation}
  I(k,\theta):=
  \begin{cases}
    \sum_{x,y\in\X} \theta_{xy} \phi^{xy*}\big(\mfrac{k^{xy}}{\theta_{xy}}\big) + \sum_{x,y \in \X} s\big(\theta_{xy} \mid (e^1\# \theta)_x P_{xy}\big),\\
             \hspace{18em} k \ll \theta \,\&\, e^1\#\theta=e^2\#\theta,\\
    \infty,  \hspace{16.7em} \text{otherwise},
  \end{cases}
  \label{eq:formal LDP2}
\end{equation}
where $P_{xy}$ is the transition probability of $(X_m)_{m\in\NN_0}$, marginals are denoted by $(e^1\#\theta)_x:=\sum_{y\in\X}\theta_{xy}$, $(e^2\#\theta)_y:=\sum_{x\in\X}\theta_{xy}$ and $k\ll\theta$ means that $k^{xy}=0$ whenever $\theta_{xy}=0$.
\label{th:main result}
\end{theorem}
A few comments on the assumptions are in place. First, $(X_m)_m$ needs to be a homogeneous Markov process, and but $(A_m)_m$ only needs to a hidden Markov process. Indeed, the independence assumption~\eqref{asseq:indep} implies that $(X_m,A_m)_m$ must be a homogeneous Markov process. The other way around, the independence assumption~\eqref{asseq:indep} is easily checked by showing that $(X_m,A_m)_m$ is a homogeneous Markov process where the transition rates do not depend on $A_m$. Second, the superlinear growth assumption~\eqref{asseq:superlinear} is really needed to ensure goodness of the mappings $a\mapsto\theta_{xy}\phi^{xy*}(\theta_{xy}^{-1}a )$ as $\theta_{xy}\to0$. 

Note that since we work in finite dimensions the topology is not an issue. The proof will be fairly straight forward: the first sum in the rate functional~\eqref{eq:formal LDP2} arises from a reweighted Cram{\'e}r's Theorem, and the second one is the known large-deviation rate for the pair-empirical measure $\Theta^n$. Nevertheless the result is interesting and relevant, as one obtains alternative, previously unknown expressions for the two classic cases in one go.


\subsection*{Application to the classic cases in continuous time}

Applying the discrete-time result to the classic cases discussed above yields the following alternative formulations for~\eqref{eq:DVG} and \eqref{eq:BFG}, in terms of inf-convolutions. 

\begin{corollary}[(A) Occupation measure LDP] Fix a $T_0>0$ and let $(X(t))_{t\geq0}$ be an irreducible continuous-time Markov chain on a finite state space $\X$ with transition probability $p_{T_0}(x,y)$. Then the large-deviation rate functional~\eqref{eq:DVG} corresponding to the occupation measure $T^{-1}\int_0^T\!\mathds1_{X(t)}\,dt$ has the alternative formulation:
\begin{align}
  &I_\mathrm{DVG}(\rho)=\inf_{\substack{\theta\in\P(\X\times\X)\\e^1\#\theta=e^2\#\theta}} \, \inf_{\substack{k\in\P(\X)^{\X\times\X}:\\ \sum_{x,y\in\X}k^{xy}=\rho}}  \, I(k,\theta), \notag\\[-1em]
\intertext{where $I$ is given by \eqref{eq:formal LDP2}, $\phi^{xy*}$ by \eqref{eq:q phi phi*}, and}
  &q^{xy}(d\rho)=\PP\Big({\textstyle\frac{1}{T_0}\int_0^{T_0}\!\mathds1_{X(t)}\,dt\in d\rho } \mid X(0)=x, X(T_0)=y\Big).
  \label{eq:q ergodic average}
\end{align}
\label{cor:DVG}
\end{corollary}

\begin{corollary}[(B) Average flux LDP] Fix a $T_0>0$, let $(X(t))_{t\geq0}$ be a continuous-time Markov chain on a finite state space $\X$ with positive jump rates $Q_{xy}>0$ between all states, and let $W_{xy}(t)$ be the cumulative flux~\eqref{eq:cumflux}. Then the large-deviation rate functional~\eqref{eq:BFG} corresponding to the pair $(T^{-1}\int_0^T\!\mathds1_{X(t)}\,dt,T^{-1}W(T))$ has the alternative formulation:
\begin{align}
  &I_\mathrm{BFG}(\rho,j)=\inf_{\substack{\theta\in\P(\X\times\X)\\e^1\#\theta=e^2\#\theta}}\, \inf_{\substack{k\in(\P(\X)\times\R^{\X\times\X})^{\X\times\X}\\ \sum_{x,y\in\X}k^{xy}=(\rho,j)}}\, 
  I(k,\theta),\notag\\[-1em]
\intertext{where $I$ is given by \eqref{eq:formal LDP2}, $\phi^{xy*}$ by \eqref{eq:q phi phi*}, and}  
  &q^{xy}(d\rho,dj)=\PP\Big({\textstyle\frac{1}{T_0}\int_0^{T_0}\!\mathds1_{X(t)}\,dt\in d\rho, \textstyle\frac{1}{T_0} W(T_0)\in dj} \mid X(0)=x, X(T_0)=y\Big).
\label{eq:q average flux}
\end{align}
\label{cor:BFG}
\end{corollary}

\subsection*{Generalisations}

To keep the notation simple we work with homogeneous Markov chains $(X(t))_{t\geq0}$. However, the discrete-time chain~\eqref{eq:discrete-time process} remains homogeneous if the continuous-time chain $(X(t))_{t\geq0}$ has $T_0$-periodic jump rates $Q(t)$, and Corollaries~\ref{cor:DVG} and \ref{cor:BFG} remain true as long as those jump rates are bounded (from above, and below away from zero). In fact, the continuous-time large-deviation principle~\eqref{eq:BFG} was already extended to periodic rates in~\cite{BCFG2018}.

The discrete-time Theorem~\ref{th:main result} can be generalised to infinite dimensions using the Dawson-G{\"a}rtner Theorem~\cite[Th.~4.6.1]{DemboZeitouni09}. More precisely, $\RR^d$ can be replaced by a dual Banach space $\A$ that has a predual, and $\X$ can be generalised to a measurable space, so that $\Theta^n$ becomes a probability measure in $\P(\X)$ and $K^n$ becomes a Banach-valued vector measures in $\M(\X\times\X;\A)$, see~\cite{Dinculeanu2000}. The exponential tightness argument of Lemma~\ref{lem:unif exp tight} is still applicable; thus a priori one obtains exponential tightness in the \emph{vague} topology of $\M(\X\times\X;\A)\times\P(\X)$. This is a real issue, because both Cram{\'e}r's Theorem as well as the pair-empirical measure large deviations are known to fail in infinite-dimensional measure spaces when equipped with the narrow topology.

\section{Proof of the result in discrete time}
\label{sec:finite}

\begin{proof}[Proof of Theorem~\ref{th:main result}] The claim will follow immediately from a mixing argument~\cite[Th.~5]{Biggins2004} after checking the following properties.
\begin{enumerate}
\item The random variable $\Theta^n$ takes values in the compact set $\P(\X\times\X)$ and is thus exponentially tight. In Lemma~\ref{lem:unif exp tight} we show the uniform exponential tightness of $K^n$, conditioned on $\Theta^n=\theta^n$ for arbitrary converging sequences $\P_n(\X\times\X)\ni\theta^n\to\theta\in\P(\X\times\X)$ from the set
\begin{equation}
  \P_n(\X\times\X):= \P(\X\times\X)\cap(n^{-1}\NN)^{\X\times\X}.
\label{eq:Pn}
\end{equation}
Together, this implies that $(K^n,\Theta^n)$ is exponentially tight~\cite[Prop.~6]{Biggins2004}.
\item The random variable $\Theta^n$ satisfies the large-deviation principle in $\P(\X\times\X)$ with good rate functional~\cite[Th.~IV.3]{Hollander08}:
\begin{align*}
  \theta\mapsto
  \begin{cases}
    \sum_{x,y \in \X} s\big(\theta_{xy} \mid (e^1\#\theta)_x P_{xy}\big), &e^1\#\theta=e^2\#\theta,\\
    \infty,   &\text{otherwise}.
  \end{cases}
\end{align*}
\item In Lemma~\ref{prop:cond LDP} we show that for arbitrary converging sequences $\P_n(\X\times\X)\ni\theta^n\to\theta\in\P(\X\times\X)$, the random variables $K^n$ conditioned on $\Theta^n=\theta^n$ satisfies the large-deviation principle with rate functional
\begin{equation}
  I_\cond(k\mid\theta):=
   \begin{cases}
    \sum_{x,y\in\X} \theta_{xy} \phi^{xy*}\big(\mfrac{k^{xy}}{\theta_{xy}}\big), &k \ll \theta,\\
    \infty,                                                                       &\text{otherwise},
  \end{cases}
\label{eq:Cramer RF1}
\end{equation}
where we implicitly set $\theta_{xy}\phi^{xy*}(k^{xy}/\theta_{xy}):=0$ whenever $\theta_{xy}=0$ and $k^{xy}=0$. Here it is essential that we chose $\P_n(\X\times\X)$ so that $\theta^n_{xy}>0$ for all $x,y$.
\item Finally, $I_\cond(k\mid\theta)=\sum_{x,y\in\X} \sup_{\lambda^{xy}\in } \lbrack k^{xy}\cdot \lambda^{xy}-\theta_{xy}\phi^{xy}(\lambda^{xy})\rbrack$ is clearly jointly lower semicontinuous in $k$ and $\theta$.

\end{enumerate}
\end{proof}

In the following, the key will be to observe that both $K^n$ and $\Theta^n$ are invariant under permutations of the indices $m=1,\hdots,n$, so that the independence assumption~\eqref{asseq:indep} yields, for box sets $dk=\bigtimes_{x,y\in\X}dk^{xy}$:
\begin{align}
  \PP(K^n\in dk\mid \Theta^n=\theta^n)
  &= \prod_{x,y\in\X} \PP\big({\textstyle \frac1n\sum_{m=1}^{n\theta^n_{xy}} \tilde A^{xy}_{m}\in dk^{xy} }\big),
\label{eq:Cramer app}
\end{align}
where the new variables $\tilde A^{xy}_{m}$ are independent identically distributed with probability $q^{xy}$.

\begin{lemma} For each pair $x,y\in\X$ the sequence of conditional probabilities $\PP(K^{n,xy} \in da \mid \Theta^n=\theta^n)$ is uniformly exponentially tight, i.e. for each $\eta>0$ there exists an $R_\eta>0$ such that for all convergent sequences
$\P_n(\X\times\X) \ni \theta^n\rightharpoonup\theta\in\P(\X\times\X)$,
\begin{equation*}
  \limsup_{n\to\infty}\frac1n\log\PP\big(\lvert K^{n,xy}\rvert_1 >R_\eta \bigmid \Theta^n=\theta) \leq -\eta.
\end{equation*}
As a consequence, $K^n$ is also uniformly exponentially tight.
\label{lem:unif exp tight}
\end{lemma}
\begin{proof} By \eqref{eq:Cramer app}, using a standard Chernoff bound and then $\theta_{xy}^n\leq1$:
\begin{align}
  \limsup_{n\to\infty}\frac1n\log\PP\big(\lvert K^{n,xy}\rvert_1 \geq R_\eta \mid \Theta^n=\theta\big) &\leq \limsup_{n\to\infty}-\sup_{s\in\RR}\lbrack sR_\eta-\theta^n_{xy}\phi_{\lvert\cdot\rvert}^{xy}(s)\rbrack \label{eq:Chernoff}\\
  &\leq -\phi_{\lvert\cdot\rvert}^{xy*}(R_\eta). \notag
\end{align}
The result then follows from Assumption~\eqref{asseq:superlinear}.
\end{proof}

\begin{lemma} Take an arbitrary sequence $\P_n(\X\times\X) \ni \theta^n\to\theta\in\P(\X\times\X)$. Then the sequence of conditional probabilities $\PP(K^n\in dk\mid \Theta^n=\theta^n)$ satisfies the large-deviation principle with good rate functional $I_\cond(k\mid\theta)$.
\label{prop:cond LDP}
\end{lemma}

\begin{proof} By the independence~\eqref{eq:Cramer app} we can show the large-deviation principle for each pair $x,y$ separately. We distinguish between three cases.

Let $x,y\in\X$ be a pair for which $\theta_{xy}>0$. By Cram{\'e}r's Theorem~\cite[Cor.~6.1.6]{DemboZeitouni09} 
\footnote{In fact, that particular version of Cramér's Theorem requires an additional condition to derive goodness of the rate functional, but this follows from our Assumption~\eqref{asseq:superlinear} since $\phi^{xy*}(a)\geq\phi_{\lvert\cdot\rvert}^{xy*}(\lvert a \rvert_1)$.}
the sequence $(n\theta^n_{xy})^{-1}\sum_{m=1}^{n\theta_{xy}^n}\tilde A^{xy}_{m}$ satisfies the large-deviation principle \emph{with speed} $n\theta_{xy}^n\to\infty$ and good rate functional $\phi^{xy*}(k^{xy})$. The sequence $n^{-1}\sum_{m=1}^{n\theta^n_{xy}} \tilde A^{xy}_{m}$ thus satisfies the large-deviation principle \emph{with speed} $n$ and good rate functional $\theta_{xy} \phi^{xy*}\big(\mfrac{k^{xy}}{\theta_{xy}}\big)$.

Now let $x,y\in X$ be a pair for which $\theta_{xy}=0$ but $k^{xy}\neq0$. 
For arbitrary $\epsilon>0$ the Chernoff bound~\eqref{eq:Chernoff}, $0<\theta_{xy}^n\to\theta_{xy}=0$ and Assumption~\eqref{asseq:superlinear} together yield:
\begin{align}
  \limsup_{n\to\infty}\frac1n\log\PP\big(\lvert K^{n,xy}\rvert_1 \geq \epsilon) \leq 
-\liminf_{n\to\infty} \theta^n_{xy}\phi_{\lvert\cdot\rvert}^{xy*}(\tfrac{\epsilon}{\theta^n_{xy}})=-\infty.
\label{eq:one xy exp tight}
\end{align}
This shows that $I_\cond(k|\theta)=\infty$ whenever $k\ll\theta$ is violated.

Finally, consider a pair for which $0<\theta^n_{xy}\to\theta_{xy}=0$ and $k^{xy}=0$. Since $K^{n,xy}$ is exponentially tight by Lemma~\ref{lem:unif exp tight}, the corresponding rate functional for that particular pair $x,y$ must have infimum zero. We conclude that $\theta_{xy}\phi^{xy*}(k^{xy}/\theta_{xy})=0$.
\end{proof}


\section{Proof for the classical cases in continuous time}

Now $(X(t))_{t\geq0}$ is a given \emph{continuous-time} Markov chain on finite state space $\X$, and to keep notation tidy let us restrict to homogeneous chains. In order to prove Corollaries~\eqref{cor:DVG} and \eqref{cor:BFG}, the assumptions of Theorem~\ref{th:main result} need to be checked for the two specific settings; the results then follow immediately from the Contraction Principle~\cite[Th.~4.2.1]{DemboZeitouni09}.

The second result requires more work than the first one. Unfortunately it is generally difficult to obtain an explicit expression for $q^{xy}$. However, the following bridge representation will be helpful.  As before, the transition probability and generator matrices of the unconditioned chain are denoted by $P(t)$ and $Q$.
\begin{lemma} Conditioned on $X(0)=x, X(T_0)=y$, the process $(X(t))_{t\in\lbrack0,T_0\rbrack}$ is an inhomogeneous Markov chain with transition probabilities, for $0< s\leq t< T_0$,
\begin{align*}
  P_{ab}^{xy}(s,t)&:=\PP\big(X(t)=b\,\mid\, X(0)=x, X(s)=a, X(T_0)=y\big)\\
    &=\frac{P_{ab}(t-s)P_{by}(T_0-t)}{P_{ay}(T_0-s)},
\intertext{and the corresponding time-dependent generator matrix is given by:}
  Q_{ab}^{xy}(t)&= \frac{\big(P(t-s)Q\big)_{ab} P_{by}(T_0-t)-\mathds1_{ab}\big(Q P(T_0-t)\big)_{by}}{P_{ay}(T_0-t)}.
\end{align*}
\end{lemma}
The Markovianity follows from Doob's h-transform~\cite{fitzsimmons1992markovian}; the transition probability follows trivially from the Markov property of the unconditioned chain, and the generator is derived from the transition probability using the forward and backward Kolmogorov equations. Because of the Markov property of the new, conditioned chain the transition probabilities $P_{ab}^{xy}(s,t)$ do not depend on $x$, but we shall keep the $x$ in the superindex for consistency of notation.

\paragraph{(A) Occupation measure LDP.} 
In this setting $A(t):=\mathds1_{X(t)}$ in $\RR^d:=\RR^\X$, and the discrete-time process $(X_m,A_m)_m$ is defined by \eqref{eq:discrete-time process}. Let $q,\phi,\phi^*$ be given by \eqref{eq:q ergodic average} and \eqref{eq:q phi phi*}.
\begin{proof}[Proof of Corollary~\ref{cor:DVG}] The coupled process $(X(t),{T_0}^{-1}\int_0^t\mathds1_{X(s)}\,ds)_{t\geq0}$ is also Markovian, with generator
\begin{align*}
  (\tilde\Q f)(x,\rho)=\sum_{y\in \X} Q_{xy}\lbrack f(y,\rho)-f(x,\rho)\rbrack + \tfrac1{T_0}\mathds1_x\cdot\grad_\rho f(x,\rho).
\end{align*}
Therefore $(X_m,A_m)_m$ is Markovian, and since the jump and drift rates in the generator do not depend on $\rho$, the conditional independence~\eqref{asseq:indep} holds.

For the growth condition~\eqref{asseq:superlinear}: $\tilde A^{xy}_{m} \in \P(\X)$ almost surely, implying that $\phi_{\lvert\cdot\rvert}^{xy*}(r)=\infty$ whenever $r\notin\lbrack0,1\rbrack$.
\end{proof}

\paragraph{(B) Average flux LDP.}
Recall that $W(t)$ is the cumulative empirical flux~\eqref{eq:cumflux} and $A(t):=(\mathds1_{X(t)},\dot W(t))$, which lies in $\RR^d:=\RR^\X\times\RR^{\X\times\X}$. Define the discrete-time process $(X_m,A_m)_m$ by \eqref{eq:discrete-time process}, and $q,\phi,\phi^*$ by \eqref{eq:q average flux} and \eqref{eq:q phi phi*}.
\begin{proof}[Proof of Corollary~\ref{cor:BFG}] 
Similar to the argument above, the coupled process\break $(X(t),T_0^{-1}\int_0^t\mathds1_{X(s)}\,ds,T_0^{-1}W(t))_{t\geq0}$ is Markovian with generator
\begin{align*}
  (\tilde\Q f)(x,\rho,j)=\sum_{y\in \X} Q_{xy}\lbrack f(y,\rho,j+\tfrac1T\mathds1_{xy})-f(x,\rho,j)\rbrack + \tfrac1T\mathds1_x\cdot\grad_\rho f(x,\rho,j),
\end{align*}
and since the jump and drift rates do not depend on $(\rho,j)$, the conditional independence~\eqref{asseq:indep} holds.

To check the growth condition~\eqref{asseq:superlinear} for each pair $x,y\in\X$, we use $\lvert T_0^{-1}\int_0^{T_0}\!\mathds1_{X(t)}\,dt\rvert_1=1$ and replace the conditional process $(\lvert T_0^{-1}W(t)\rvert_1)_{t\in\lbrack0,T_0\rbrack}$ by a simpler, real-valued process $(U(t))_{t\in\lbrack0,T_0\rbrack}$ with higher jump rates. To this aim, note that for any pair $a\neq b\in\X$:
\begin{align*}
  \lim_{t\to 0} Q_{ab}^{xy}(t)
  =
  \frac{Q_{ab}P_{by}(T_0)}{P_{ay}(T_0)}
  &&
  \lim_{t\to T_0} Q_{ab}^{xy}(t)
  =
  \begin{cases}
    \frac{Q_{ab}Q_{by}}{Q_{ay}}, &a,b\neq y,\\
    Q_{yb}P_{by}(T_0), &a=y,\\
    \infty, &b=y,\\
  \end{cases}
\end{align*}
so that by continuity $\bar{Q}_{ab}^{xy}:=\sup_{t\in\lbrack0,T_0\rbrack} Q_{ab}^{xy}(t)<\infty$ for $b\neq y$ (we assumed that all $Q_{ab}>0$). The process $(U(t))_{t\in\lbrack0,T_0\rbrack}$ with $U(0)=0$ makes independent jumps $+T_0^{-1}$ with Poisson rates $\bar{Q}_{ab}^{xy}$ for $a\neq b, a,b\neq y$, and jumps $+2T_0^{-1}$ with rate $\bar{Q}_{yb}^{xy}$ for $b\neq y$; the factor $2$ represents an \emph{instantaneous} jump back to $y$. Then, setting $\bar{Q}^{xy}:=\sum_{\substack{a,b\in\X,\\a\neq b, a,b\neq y}} \bar{Q}^{xy}_{ab} + \sum_{\substack{b\in\X,\\b\neq y}}\bar{Q}^{xy}_{yb}$,
\begin{align*}
  \phi^{xy}_{\lvert\cdot\rvert}(s)\leq\log\EE e^{s(1+U(T_0))}\leq s+\bar{Q}^{xy} T_0(e^{2s/T_0}-1). 
\end{align*}
From this one obtains the claimed superlinear growth:
\begin{align*}
  \phi^{xy*}_{\lvert\cdot\rvert}(r) \geq \frac12 s(r-1 \, \mid\, 2\bar{Q}^{xy}).
\end{align*}

\end{proof}

\bibliographystyle{alpha}	
\bibliography{lib}

\end{document}